\documentclass[12pt]{amsart}

\author{Paul \textsc{Poncet}}
\address{CMAP, \'{E}cole Polytechnique, Route de Saclay, 91128 Palaiseau Cedex, France \\
and INRIA, Saclay--\^{I}le-de-France}

\email{poncet@cmap.polytechnique.fr}

\usepackage{stmaryrd}
\usepackage{amssymb}
\usepackage{amsmath}
\usepackage{t1enc}
\usepackage{yhmath}
\usepackage{mathabx}
\usepackage{calrsfs} 
\usepackage{times} 
\DeclareMathOperator*{\id}{id}
\DeclareMathOperator*{\dom}{dom}

\newtheorem{theorem}{Theorem}[section]
\newtheorem{corollary}[theorem]{Corollary}
\newtheorem{proposition}[theorem]{Proposition}
\newtheorem{lemma}[theorem]{Lemma}

\theoremstyle{definition}
\newtheorem{definition}[theorem]{Definition}
\newtheorem{example}[theorem]{Example}
\newtheorem{examples}[theorem]{Examples}
\newtheorem{problem}[theorem]{Problem}

\newenvironment{acknowledgements}[1][]{\par\vspace{0.5cm}\noindent\textbf{Acknowledgements#1.} }{\par}

\begin{document}

\title{Domain theory and mirror properties\\in inverse semigroups}

\date{\today}

\subjclass[2010]{20M18, 
                 06A12, 
                 06B35, 
                 06F05} 

\keywords{inverse semigroups, domains, continuous posets, algebraic posets, meet-continuous posets, multiplicative way-above relation}

\begin{abstract}
Inverse semigroups are a class of semigroups whose structure induces a compatible partial order. This partial order is examined so as to establish  \textit{mirror properties} between an inverse semigroup and the semilattice of its idempotent elements, such as continuity in the sense of domain theory. 
\end{abstract}

\maketitle

\section{Introduction}

The branch of order theory called \textit{domain theory} was initiated in the early 1970's with the pioneering work of Dana S.\ Scott on a model of untyped lambda-calculus \cite{Scott72}. 
Of special interest among the class of domains are \textit{continuous semilattices}. This is partly due to the ``fundamental theorem of compact semilattices'', which identifies continuous complete semilattices and compact topological semilattices with small semilattices (see Jimmie D.\ Lawson \cite{Lawson69, Lawson73} for an early form of this result, Hofmann and Stralka \cite{Hofmann76} for the original statement, Lea \cite{Lea76} for an alternative proof; see also Gierz et al.\ \cite[Theorem~VI-3.4]{Gierz03}). 

A generalization of semilattices is the concept of \textit{inverse semigroup}; it dates back to the 1950's with the works of Wagner \cite{Vagner52}, Liber \cite{Liber53}, and Preston \cite{Preston54a}. 
An inverse semigroup is naturally endowed with a compatible partial order (called \textit{intrinsic}), and many authors have investigated its structure from the point of view of order theory, including Wagner \cite{Vagner52}, Mitsch \cite{Mitsch78}, Mark V.\ Lawson \cite{Lawson98b}, Resende \cite{Resende06}. 
See also Mitsch \cite{Mitsch86, Mitsch94} for the extension of this partial order to \textit{every} semigroup (continuing the works of Hartwig \cite{Hartwig80} and Nambooripad \cite{Nambooripad80} on regular semigroups). 

Inverse semigroups also form a nice generalization of groups. While many tools of group theory have been successfully exported to inverse semigroup theory, the  contribution of semilattice theory is barely visible. 
Especially, no attempt has been made to apply the concepts (of approximation and continuity in particular) of domain theory to the framework of inverse semigroups. The purpose of this work is to fill this gap. More precisely, we aim at proving what we call \textit{mirror properties}, i.e.\ properties that hold for an inverse semigroup $S$ if and only if they hold for its semilattice $\mathit{\Sigma}(S)$ of idempotent elements. Our main theorem asserts that continuity and algebraicity in the sense of domain theory are mirror properties:

\begin{theorem}
Assume that $S$ is a mirror semigroup. Then $S$ is continuous (resp.\ algebraic) if and only if $\mathit{\Sigma}(S)$ is continuous (resp.\ algebraic). 
\end{theorem}

The hypothesis of being a \textit{mirror} semigroup is a simple technical condition on suprema of directed subsets. 
We apply this result to a series of examples such as the symmetric pseudogroup of a topological space, the bicyclic monoid, or the semigroup of characters of an inverse semigroup. In passing we redefine the notion of character to encompass both the classical one used in group theory and the one that plays this role in semilattice theory (where ``characters'' take their values in $[0,1]$ rather than in the complex disc). 

This study of inverse semigroups from the point of view of domain theory is motivated by the recent work of Castella \cite{Castella10}. He realized that inverse semigroups (that he rediscovered under the name of \textit{quasigroups}) are a \textit{good} (i.e.\ not a too strong, unlike general semigroup theory) generalization of groups and semilattices. He applied this concept to the study of \textit{inverse semirings}, \textit{inverse semifields} and polynomials with coefficients in these semifields. He thus sketched a global theory gathering both classical algebra and max-plus algebra. 
This echos with other very recent work by Lescot \cite{Lescot09}, Connes and Consani \cite{Connes12}, and Connes \cite{Connes10}. These authors looked for the ``one-element field'' $\mathbb{F}_1$ which is an undefined concept introduced in geometry by Tits \cite{Tits57}. This quest led them to max-plus algebra and tropical geometry, and the semilattices they considered, once endowed with an additional binary relation, are seen as semifields of characteristic one, hence as a particular case of the whole class of semifields of non-zero characteristic.  

Little by little, a new field of investigation takes shape, not only connected with algebraic questionings, but also with analytic ones. It would aim at unifying classical mathematics (where addition predominates) and idempotent mathematics (based on the maximum operation). But such a program will not be effective without domain theory. For even if domains are of limited importance in classical mathematics (and inverse semigroup theory tells us that the intrinsic order reduces to equality when groups are at stake), it becomes a crucial tool in idempotent mathematics, as Jimmie D.\ Lawson explained in \cite{Lawson04b}. We also observed this in \cite[Chapters~II, IV, V]{Poncet11}: proving idempotent analogues of the Krein--Milman theorem or the Choquet integral representation theorem necessitates such a theory. 

The paper is organized as follows. 
Section~\ref{sec:is} gives some basics of inverse semigroup theory and recalls or builds some key examples. 
In Section~\ref{sec:compl} we focus on completeness properties of inverse semigroups, recalling some known results due to Rinow \cite{Rinow63} and Domanov \cite{Domanov71} 
and bringing up some new ones. 
In Section~\ref{sec:meetc} 
we prove mirror properties related to separate Scott-continuity / meet-continuity. 
Our main theorem on continuous (resp.\ algebraic) inverse semigroups is the purpose of Section~\ref{sec:cont}.

\section{Preliminaries on inverse semigroups}\label{sec:is}

A \textit{semigroup} $(S, \cdot)$ is a set $S$ equipped with an associative binary relation. An element $\epsilon$ of $S$ is \textit{idempotent} if $\epsilon \epsilon = \epsilon$. The set of idempotent elements of $S$ is denoted by $\mathit{\Sigma}(S)$. 
The semigroup $S$ is \textit{inverse} if the idempotents of $S$ commute and if, for all $s \in S$, there is some $t \in S$, called an \textit{inverse} of $s$, such that $s t s = s$ and $t s t = t$. An \textit{inverse monoid} is an inverse semigroup with an identity element. 
It is well-known that a semigroup is inverse if and only if every element $s$ has a unique inverse, denoted by $s^*$. Our main reference for inverse semigroup theory is the monograph by Mark V.\ Lawson \cite{Lawson98b}. The reader may also consult Petrich's book \cite{Petrich84}. 

It is worth recalling some basic rules for inverse semigroups. 

\begin{proposition}	
Let $S$ be an inverse semigroup, and let $s, t \in S$. Then  
\begin{enumerate}
	\item $s s^*$ and $s^* s$ are idempotent. 
	\item $(s^*)^* = s$ and $(s t)^* = t^* s^*$. 
	\item $s^* = s$ if $s$ is idempotent. 
\end{enumerate}
\end{proposition}

A partial order $\leqslant$ can be defined on the idempotents of a semigroup as follows: $\epsilon \leqslant \phi$ if $\epsilon \phi = \phi \epsilon = \epsilon$. 
For an inverse semigroup, this partial order can be naturally extended to the whole underlying set, if we define $s \leqslant t$ by $s = t \epsilon$ for some idempotent $\epsilon$ \cite[Proposition~1-7]{Lawson98b}. We shall refer to $\leqslant$ as the \textit{intrinsic (partial) order} of the inverse semigroup $S$. In this case, $(\mathit{\Sigma}(S), \leqslant)$ is a \textit{semilattice} \cite[Proposition~1-8]{Lawson98b}, i.e.\ a partially ordered set in which every nonempty finite subset has an infimum (or, equivalently, a commutative idempotent semigroup). 
Also, the intrinsic order is compatible with the structure of semigroup, in the sense that $s \leqslant t$ and $s' \leqslant t'$ imply $s s' \leqslant t t'$ \cite[Proposition~1-7]{Lawson98b}. 
We recall some equivalent characterizations of $\leqslant$. 

\begin{lemma}\label{lem:leq}
Let $S$ be an inverse semigroup, and let $s, t \in S$. Then $(s \leqslant t) \Leftrightarrow (s^*\leqslant t^*) \Leftrightarrow (s = t s^* s) \Leftrightarrow (s = \epsilon t \mbox{ for some idempotent } \epsilon) \Leftrightarrow (s = s s^* t)$. 
\end{lemma}

\begin{proof}
See e.g.\ \cite[Lemma~1-6]{Lawson98b}. 
\end{proof}

The following examples are extracted from the literature, with the exception of the last two of them. 

\begin{example}[Groups]\label{ex:groups}
An inverse semigroup is a group if and only if its intrinsic order coincides with equality. The only idempotent element is the identity of the group. 
\end{example}

\begin{example}[Symmetric pseudogroup]\label{ex:sym}
Let $X$ be a nonempty set. The \textit{symmetric pseudogroup} $\mathrsfs{I}(X)$ on $X$ is the set made up of all the partial bijections on $X$, i.e.\ the bijections $f : U \rightarrow V$ where $U$ and $V$ are subsets of $X$ (in this situation we write $\dom(f)$ for $U$). This is an inverse semigroup when endowed with the composition defined by $f_1 \cdot f : x \mapsto f_1(f(x)) : f^{-1}(V \cap U_1) \rightarrow f_1(V \cap U_1)$, where $f : U \rightarrow V$ and $f_1 : U_1 \rightarrow V_1$. The involution is the inversion, given by $f^* = f^{-1} : V \rightarrow U$. Idempotent elements of $\mathrsfs{I}(X)$ are of the form $\id_U$ for some subset $U$, and the partial order on $\mathrsfs{I}(X)$ is given by 
$
f \leqslant g 
$ 
if and only if  $f = g|_U$ for some subset $U \subset \dom(g)$. 
\end{example}

\begin{example}[Cosets of a group, see e.g.\ McAlister \cite{McAlister80}]\label{ex:cosets}
Let $G$ be a group. A \textit{coset} of $G$ is a subset of the form $H g$, where $H$ is a subgroup of $G$ and $g \in G$. 
Every nonempty intersection of cosets is a coset, so, for all cosets $C, C'$, we can consider the smallest coset $C \otimes C'$ containing $C C'$. The product $\otimes$ makes the collection $\mathrsfs{C}(G)$ of cosets of $G$ into an inverse monoid called the \textit{coset monoid} of $G$. 
The intrinsic order of $\mathrsfs{C}(G)$ coincides with reverse inclusion, and the idempotents of $\mathrsfs{C}(G)$ are exactly the subgroups of $G$. 
\end{example}

\begin{example}[Bicyclic monoid, see e.g.\ Mark V.\ Lawson \protect{\cite[Sec.\ 3.4]{Lawson98b}}]\label{ex:bic}
Let $P$ be the positive cone of a lattice-ordered group. 
On $P \times P$, one can define the binary relation $\oplus$ by $(a, b) \oplus (c, d) = (a - b + b \vee c, d - c + b \vee c)$. This makes $P \times P$ into an inverse monoid with identity $(0,0)$, called the \textit{bicyclic monoid}, such that $(a, b)^* = (b, a)$. An element $(a, b)$ is idempotent if and only if $a = b$, and the intrinsic order satisfies $(a, b) \leqslant (c, d)$ if and only if $c = d + a - b$ and $d \leqslant b$. In particular, for idempotents, $(a, a) \leqslant (b, b)$ if and only if $b \leqslant a$. 
\end{example}

\begin{example}[Rotation semigroup]\label{ex:rot}
On the unit disc $B^2$ of the complex numbers $\mathbb C$, 
let us consider the binary relation defined by 
$$
z \otimes z' = \min(r, r') \exp(\mathbf{i} (\theta + \theta') ),  
$$
if one write $z = r e^{\mathbf{i} \theta}$ and $z' = r' e^{\mathbf{i} \theta'}$ with $r, r' \in [0, 1]$. 
Then $(B^2, \otimes)$ is an inverse monoid, $z^*$ coincides with the conjugate $\bar{z}$ of $z$, and $z$ is idempotent if and only if $z \in [0, 1]$. Moreover, the intrinsic order satisfies $z \leqslant z'$ if and only if $r = 0$ or ($r \leqslant r'$ and $\theta \equiv \theta' [2 \pi]$). 
\end{example}

\begin{example}[Characters]\label{ex:cha}
Let $S$ be a commutative inverse monoid. We define a \textit{character} on $S$ as a morphism $\chi : S \rightarrow B^2$ 
of inverse monoids, where $B^2$ is equipped with the inverse monoid structure of the previous example. If $1$ denotes the identity of $S$, this means that $\chi(1) = 1$ and $\chi(s t) = \chi(s) \otimes \chi(t)$ for all $s, t \in S$ (the fact that $\chi(s^*) = \chi(s)^*$ then automatically holds). This definition differs from the one used e.g.\ by Warne and Williams \cite{Warne61} and Fulp \cite{Fulp71}, for these authors equipped $\mathbb C$ or $B^2$ with the usual multiplication. 
The set $S^{\wedge}$ of characters on $S$ has itself a natural structure of inverse monoid. 
A character $\chi$ is idempotent in $S^{\wedge}$ if and only if it is $[0,1]$-valued.  
Moreover, if $S$ is actually a group (resp.\ a semilattice), then every non-zero $\chi$ takes its values into the unit circle (resp.\ into $[0,1]$). 
\end{example}

Henceforth, $S$ denotes an inverse semigroup and we let $\mathit{\Sigma} = \mathit{\Sigma}(S)$. 
The purpose of the next section is to investigate completeness properties of inverse semigroups with respect to their intrinsic order.

\section{Completeness properties of inverse semigroups}\label{sec:compl}

Of particular usefulness in the framework of domain theory is the property of being directed-complete. 
The term \textit{poset} is an abbreviation for ``partially ordered set''. 
A subset $D$ of a poset is \textit{directed} if $D$ is nonempty and, for each pair $s, t \in D$, there exists some $r \in D$ such that $s \leqslant r$ and $t \leqslant r$. The poset is \textit{directed-complete} if every directed subset has a supremum. It is \textit{conditionally directed-complete} if every principal ideal $\{ t : t \leqslant s \}$ is a directed-complete poset, i.e.\ if every directed subset bounded above has a supremum. 

In the framework of inverse semigroups, a subset of the semilattice of idempotents $\mathit{\Sigma}$ may have a supremum in 
$\mathit{\Sigma}$ but no supremum in the whole inverse semigroup. However, this property will appear desirable for establishing mirror properties, hence a definition is needed. 

\begin{definition}
The inverse semigroup $S$ is a \textit{mirror semigroup} if every directed subset of $\mathit{\Sigma}$ with a supremum in $\mathit{\Sigma}$ also has a supremum in $S$. 
\end{definition}

In this case, both suprema coincide and belong to $\mathit{\Sigma}$. Indeed, if $\mathit{\Delta}$ is a directed subset of $\mathit{\Sigma}$ with supremum $\delta$ in $\mathit{\Sigma}$ and supremum $d$ in $S$, then $d \leqslant \delta$, i.e.\ $d = \delta d^* d$. As a product of idempotent elements, $d$ is idempotent itself, so that $d = \delta$. 

\begin{example}\label{ex:cex}
As an example of inverse semigroup that is not mirror, we can consider the set $S = [0, 1] \cup \{ \omega \}$ equipped with the binary relation $\otimes$ defined by $s \otimes t = \min(s, t)$ if $s, t \in [0, 1]$, $\omega \otimes s = s \otimes \omega = s$ if $s \in [0, 1)$, $\omega \otimes 1 = 1 \otimes \omega = \omega$, $\omega \otimes \omega = 1$. Then $(S, \otimes)$ is an inverse semigroup whose subset of idempotents is $\mathit{\Sigma} = [0, 1]$. Moreover, the directed subset $[0, 1)$ admits $1$ as supremum in $\mathit{\Sigma}$, but has two incomparable upper bounds in $S$, namely $1$ and $\omega$, hence has no supremum in $S$. 
\end{example}

Fortunately, mirror semigroups are rather numerous, as the following proposition shows. 

\begin{proposition}
The inverse semigroup $S$ is a mirror semigroup in any of the following cases:
\begin{enumerate}
	\item\label{perky1} $S$ projects onto $\mathit{\Sigma}$, i.e.\ there is some order-preserving map $j : S \rightarrow \mathit{\Sigma}$ such that $j\circ j = j$ and $j \leqslant \id_S$. 
	\item\label{perky2} $S$ is \textit{reduced}, i.e.\ $s \geqslant \epsilon$ and $\epsilon \in \mathit{\Sigma}$ imply $s \in \mathit{\Sigma}$. 
	\item\label{perky3} $S$ is (conditionally) directed-complete, 
	\item\label{perky5} $(S,\cdot)$ is a semilattice, i.e.\ $S$ coincides with $\mathit{\Sigma}$, 
	\item\label{perky6} $(S, \cdot)$ is a group. 
	\item\label{perky7} $S$ is finite. 
\end{enumerate}
\end{proposition}

\begin{proof}
Cases (\ref{perky3}), (\ref{perky5}) and (\ref{perky6}) are straightforward. 
For (\ref{perky1}), (\ref{perky2}), and (\ref{perky7}), let $\mathit{\Delta}$ be a directed subset of $\mathit{\Sigma}$. 
Assume that $\mathit{\Delta}$ has a supremum $\delta$ in $\mathit{\Sigma}$, and let $u \in S$ be an upper bound of $\mathit{\Delta}$. We need to show that $\delta \leqslant u$. 

(\ref{perky1}) Assume the existence of $j : S \rightarrow \mathit{\Sigma}$. For all $\alpha \in \mathit{\Delta}$, $\alpha \leqslant u$, hence $j(\alpha) = \alpha \leqslant j(u)$. Since $j(u)$ is idempotent, this is an upper bound of $\mathit{\Delta}$ in $\mathit{\Sigma}$, and we deduce that $\delta \leqslant j(u) \leqslant u$. 

(\ref{perky2}) If $S$ is reduced, then $u \in \mathit{\Sigma}$ since $\mathit{\Delta}$ is supposed nonempty, so that $\delta \leqslant u$. 

(\ref{perky7}) If $S$ is finite, then the directed subset $\mathit{\Delta}$ is finite, so $\delta \in \mathit{\Delta}$. This implies $\delta \leqslant u$. 
\end{proof}

\begin{examples}
The symmetric pseudogroup is directed-complete; the coset monoid of a group is conditionally directed-complete; the bicyclic monoid, the rotation semigroup, and the character monoid are reduced. Thus, all the examples of inverse semigroups that we introduced in the previous section are mirror semigroups. 
\end{examples}

The reader may ask why the definition of a mirror semigroup does not require that every directed subset of $\mathit{\Sigma}$ with a supremum in $S$ also has a supremum in $\mathit{\Sigma}$. This property turns out to hold in every inverse semigroup. The following lemma gives a stronger statement. If $A \subset S$, we write $\bigvee A$ for the supremum of $A$ in $S$, whenever it exists. Also, we denote by $\sigma$ the \textit{source map} $S \rightarrow \mathit{\Sigma}$ defined by $s \mapsto \sigma(s) = s^* s$. 

\begin{lemma}\cite{Rinow63}\label{lem:dd}
Let $A$ be a nonempty subset of $S$. 
If $\bigvee A$ exists, then $\bigvee \sigma(A)$ exists and $\bigvee \sigma(A) = \sigma(\bigvee A)$. 
\end{lemma}

\begin{proof}
The reader may refer to \cite[Lemma~1-17]{Lawson98b}. 
\end{proof}

As a consequence, the map $\sigma$ is Scott-continuous (see e.g.\ \cite[Proposition~~II-2.1]{Gierz03} and adapt the proof to the case of posets that are not directed-complete), and $\mathit{\Sigma}$ is a Scott-closed subset of $S$ and is a retract of $S$ when $S$ is endowed with its Scott topology. 

Another important feature of suprema that we shall need later on is a kind of conditional distributivity property. 

\begin{lemma}\cite{Domanov71}\label{lem:distr}
Let $A$ be a nonempty subset of $S$ and $s \in S$. 
If $\bigvee A$ exists and $a a^* \leqslant s^* s$ for all $a \in A$, then $\bigvee (s A)$ exists and $\bigvee (s A) = s (\bigvee A)$.  
\end{lemma}

\begin{proof}
See e.g.\ \cite[Proposition~1-18]{Lawson98b}. 
\end{proof}

As a corollary, we get our first mirror property. 

\begin{proposition}
Assume that $S$ is a mirror semigroup. Then $S$ is conditionally directed-complete if and only if $\mathit{\Sigma}$ is conditionally directed-complete. 
\end{proposition}

\begin{proof}
Assume that $\mathit{\Sigma}$ is conditionally directed-complete, and let $D$ be a directed subset of $S$ that is bounded above. Then $\mathit{\Delta} = \sigma(D)$ is bounded above and directed in $\mathit{\Sigma}$. Let $\delta = \bigvee \mathit{\Delta}$. We show that, if $u$ is an upper bound of $D$ in $S$, then $\bigvee D$ (exists and) equals $u \delta$. For every $\alpha \in \mathit{\Delta}$, $\alpha = \alpha \alpha^* \leqslant u^* u$, so by Lemma~\ref{lem:distr}, $\bigvee (u \mathit{\Delta})$ exists and equals $u \delta$. But, for every $d \in D$, $d \leqslant u$, i.e.\ $u d^* d = d$ by Lemma~\ref{lem:leq}, so that $u \mathit{\Delta} = \{ u d^* d : d \in D \} = \{ d : d \in D \} = D$. Hence $\bigvee D = u \delta$. 
\end{proof}

\section{Separate Scott-continuity and multiplicative way-below relation}\label{sec:meetc}

The second mirror property concerns separate Scott-continuity. A mirror semigroup $S$ is \textit{separately Scott-continuous} if, for all directed subsets $D \subset S$ with supremum, and for all $s \in S$, $\bigvee (D s)$ exists and equals $(\bigvee D) s$. This is tantamount to saying that the map $t \mapsto t s$ is Scott-continuous for all $s \in S$. 
Restricted to the case of semilattices, separate Scott-continuity can be called \textit{meet-continuity} (although \cite[Definition~III-2.1]{Gierz03} proposes a slightly different notion of meet-continuity). 

\begin{proposition}\label{prop:meetcont}
Assume that $S$ is a mirror semigroup. 
Then $S$ is separately Scott-continuous if and only if $\mathit{\Sigma}$ is meet-continuous. 
\end{proposition}

\begin{proof}
Since $\mathit{\Sigma}$ is a Scott-closed subset of $S$, separate Scott-continuity of $S$ clearly implies meet-continuity of $\mathit{\Sigma}$. 
For the converse statement, we mainly follow the lines of the proof of \cite[Proposition~1-20]{Lawson98b}. Assume that $\mathit{\Sigma}$ is meet-continuous. Let $D$ be a directed subset of $S$ with supremum $d$, and let $s \in S$. The set $D s$ is bounded above by $d s$. Now let $u$ be some upper bound of $D s$. By Lemma~\ref{lem:dd} we have 
$$
d^* d s s^* = \sigma(\bigvee D) s s^* = (\bigvee \sigma(D)) s s^*. 
$$
Since $\sigma(D)$ is directed in $\mathit{\Sigma}$ and $\mathit{\Sigma}$ is meet-continuous, this gives   
$$
d^* d s s^* =  \bigvee (\sigma(D) s s^*).  
$$
Now $D s$ is bounded above by $u$, thus $\sigma(D) s s^*$ is bounded above by $d^* u s^*$, so that $d^* d s s^* \leqslant d^* u s^*$. We get   
$$
d s = d (d^* d s s^*) s \leqslant d d^* u s^* s \leqslant u. 
$$
Hence, $d s = \bigvee (D s)$. 
\end{proof}

\begin{problem}[Example~\ref{ex:cosets} continued]
Is the coset monoid $\mathrsfs{C}(G)$ of a group $G$ separately Scott-continuous?
\end{problem}

Since $\mathit{\Sigma}$ is commutative, we deduce that a mirror semigroup is separately Scott-continuous if and only if, for all directed subsets $D$ with supremum, and for all $s$, $\bigvee (s D)$ exists and equals $s (\bigvee D)$. This will help in demonstrating Lemma~\ref{lem:wb}. 

The following result, although easily proved, will be of crucial importance for establishing Proposition~\ref{prop:mult} and Theorem~\ref{thm:continuous}. It highlights the role played by the \textit{subset system} made up of directed subsets for deriving mirror properties. For instance, mirror properties on complete distributivity (also called \textit{supercontinuity}), where arbitrary subsets replace directed subsets, would probably fail to be true (see however the mirror property \cite[Proposition~1-20]{Lawson98b} and the one on infinite distributivity proved by Resende \cite{Resende06}). 

\begin{lemma}\label{lem:maj}
Let $D$ be a directed subset of $S$. Then $d$ is the greatest element of $D d^* d$, for all $d \in D$. 
\end{lemma}

\begin{proof}
Let $d, d_1 \in D$. Since $D$ is directed, there is some $d_2 \in D$ such that $d \leqslant d_2$ and $d_1 \leqslant d_2$. Thus, $d_1 d^* d \leqslant d_2^{} d_2^* d \leqslant d$, so $d$ is an upper bound of $D d^* d$. Since $d$ is in $D d^* d$, the assertion is proved. 
\end{proof}

Now we recall the way-below relation. If $s, t$ are elements of a poset, we say that $s$ is \textit{way-below} $t$, 
if, for every directed subset $D$ with supremum, $t \leqslant \bigvee D$ implies $s \leqslant d$ for some $d \in D$. 
Denoting by $\ll$ the way-below relation on $S$ and by $\prec\!\!\!\prec$ the way-below relation on $\mathit{\Sigma}$, we have the following lemma. 

\begin{lemma}\label{lem:wb}
Assume that $S$ is a separately Scott-continuous mirror semigroup. Then for all $s, t \in S$, $s \ll t$ if and only if $s \leqslant t$ and $\sigma(s) \prec\!\!\!\prec \sigma(t)$. 
\end{lemma}

\begin{proof}
Assume that $s \ll t$, and let $\mathit{\Delta}$ be some directed subset of $\mathit{\Sigma}$ with supremum such that $\sigma(t) \leqslant \bigvee \mathit{\Delta}$. Then $t = t \sigma(t) \leqslant \bigvee (t \mathit{\Delta})$. Since $t \mathit{\Delta}$ is directed and $s \ll t$, there is some $\delta \in \mathit{\Delta}$ such that $s \leqslant t \delta$. This also gives $s^* \leqslant \delta t^*$, thus $\sigma(s) = s^* s \leqslant \delta \sigma(t) \delta = \sigma(t) \delta \leqslant \delta$. We therefore get $\sigma(s) \prec\!\!\!\prec \sigma(t)$. 

Conversely, assume that $s \leqslant t$ and $\sigma(s) \prec\!\!\!\prec \sigma(t)$, and let $D$ be a directed subset of $S$ with supremum such that $t \leqslant \bigvee D$. We have $\sigma(t) \leqslant \sigma(\bigvee D) =  \bigvee \sigma(D)$, and $\sigma(D)$ is directed in $\mathit{\Sigma}$, thus there is some $d\in D$ such that $\sigma(s) \leqslant \sigma(d)$. Now $s = s \sigma(s) \leqslant s \sigma(d) \leqslant t \sigma(d) \leqslant (\bigvee D) \sigma(d) = \bigvee (D d^* d)$. Using Lemma~\ref{lem:maj}, we have $s \leqslant d$, and this shows that $s \ll t$. 
\end{proof}

\begin{corollary}\label{cor:wb}
Assume that $S$ is a separately Scott-continuous mirror semigroup. Then for all $\epsilon, \delta \in \mathit{\Sigma}$, $\epsilon \prec\!\!\!\prec \delta$ if and only if $\epsilon \ll \delta$.
\end{corollary}

The way-below relation on $S$ is \textit{multiplicative} if $s \ll t$ and $s' \ll t'$ imply $s s' \ll t t'$. If $S$ reduces to a semilattice, this amounts to the usual definition (see \cite[Definition~III-5.8]{Gierz03}). 

\begin{proposition}\label{prop:mult}
Assume that $S$ is a separately Scott-continuous mirror semigroup. Then the way-below relation on $S$ is multiplicative if and only if the way-below relation on $\mathit{\Sigma}$ is multiplicative. 
\end{proposition}

\begin{proof}
If $S$ has a multiplicative way-below relation, the previous corollary ensures that $\mathit{\Sigma}$ also has a multiplicative way-below relation. 
Conversely, assume that $\prec\!\!\!\prec$ is multiplicative, and let $r, s, t, u \in S$ such that $r \ll s$ and $t \ll u$. Let $D$ be a directed subset of $S$ with supremum $d_0$ such that $s u \leqslant d_0$. Then $s^* s u u^* \leqslant s^* d_0 u^*$, and since $s^* s u u^*$ is idempotent, $s^* s u u^* \leqslant \sigma(s^* d_0 u^*)$. Since $S$ is separately Scott-continuous and $s^* D u^*$ is directed, we have $s^* d_0 u^* = \bigvee (s^* D u^*)$, and by Scott-continuity of $\sigma$ we deduce 
\begin{equation}\label{eq:eq2}
s^* s u u^* \leqslant \bigvee \sigma(s^* D u^*). 
\end{equation} 
By Lemma~\ref{lem:wb}, $r^* r \prec\!\!\!\prec s^* s$, and similarly $t t^* \prec\!\!\!\prec u u^*$. Since $\prec\!\!\!\prec$ is multiplicative, this gives $r^* r t t^* \prec\!\!\!\prec s^* s u u^*$. Combining this with Equation~(\ref{eq:eq2}), we see that there is some $d \in D$ such that $r^* r t t^* \leqslant \sigma(s^* d u^*)$. 
Hence, 
\begin{align*}
r t &= r (r^* r t t^*) t \\
&\leqslant s \sigma(s^* d u^*) u \\
&\leqslant ((s u) d^*) (s s^*) d (u^* u) \\
&\leqslant (d_0^{} d_0^*) (s s^*) d (u^* u) \\
&\leqslant d. 
\end{align*}
This proves that $r t \ll s u$, i.e.\ that $\ll$ is multiplicative. 
\end{proof}

\section{Continuity, algebraicity}\label{sec:cont}

A poset is \textit{continuous} if $\{ t : t \mbox{ way-below } s \}$ is directed with supremum equal to $s$, for all $s$. A \textit{dcpo} is a directed-complete poset, and a \textit{domain} is a continuous dcpo. 
A poset is \textit{algebraic} if every element $s$ is the directed supremum of the compact elements below it (where an element is \textit{compact} if it is way-below itself). Or equivalently, by \cite[Proposition~I-4.3]{Gierz03}, if the poset is continuous and if $t \ll s$ implies $t \leqslant k \leqslant s$ for some compact element $k$. 
A mirror semigroup is \textit{continuous} (resp.\ \textit{algebraic}) if it is continuous (resp.\ \textit{algebraic}) with respect to its intrinsic partial order. 

\begin{lemma}\label{lem:cpc}
A continuous mirror semigroup is separately Scott-continuous. 
\end{lemma}

\begin{proof}
Assume that $S$ is a continuous mirror semigroup. We show that $S$ is separately Scott-continuous, or equivalently by Proposition~\ref{prop:meetcont} that $\mathit{\Sigma}$ is meet-continuous. Let $\mathit{\Delta}$ be a directed subset of $\mathit{\Sigma}$ with supremum $\delta = \bigvee \mathit{\Delta}$, and let $\epsilon \in \mathit{\Sigma}$. Then $\epsilon \mathit{\Delta}$ is bounded above by $\epsilon \delta$. Now let $u$ be some upper bound of $\epsilon \mathit{\Delta}$ in $S$. We prove that $\epsilon \delta \leqslant u$. For this purpose, let $s \ll \epsilon \delta$. Since $s$ is less than the idempotent element $\epsilon \delta$, $s$ is idempotent itself. 
Moreover, we have $s \ll \delta$, so there exists some $\delta_1 \in \mathit{\Delta}$ such that $s \leqslant \delta_1$. Thus, $s = s s \leqslant \epsilon \delta_1 \leqslant u$. Since $S$ is continuous, we deduce that $\epsilon \delta \leqslant u$, hence $\epsilon \delta = \bigvee (\epsilon \mathit{\Delta})$, so $\mathit{\Sigma}$ is continuous, and the result follows. 
\end{proof}

\begin{lemma}\label{lem:cont}
Assume that $S$ is a mirror semigroup. 
If $S$ is a continuous poset (resp.\ a dcpo, a domain, an algebraic poset), then $\mathit{\Sigma}$ is a continuous poset (resp.\ a dcpo, a domain, an algebraic poset). 
\end{lemma}

\begin{proof}
Assume that $S$ is directed-complete. Since $\mathit{\Sigma}$ is Scott-closed in $S$, this is a sub-dcpo of $S$ by \cite[Exercise~II-1.26(ii)]{Gierz03}. 

Assume that $S$ is a continuous poset, and let $\epsilon \in \mathit{\Sigma}$. Every element way-below $\epsilon$ in $S$ belongs to $\mathit{\Sigma}$, and is way-below $\epsilon$ in $\mathit{\Sigma}$ (this merely results from the fact that $S$ is mirror). Thus, $\epsilon$ is the supremum (in $\mathit{\Sigma}$) of a directed subset of elements way-below it, which gives continuity of $\mathit{\Sigma}$. 

Assume that $S$ is algebraic. To prove that $\mathit{\Sigma}$ is algebraic, we need to show that, whenever $\epsilon \prec\!\!\!\prec \delta$, there is some $\kappa \in \mathit{\Sigma}$ with $\kappa \prec\!\!\!\prec \kappa$ and $\epsilon \leqslant \kappa \leqslant \delta$. But $S$ is separately Scott-continuous by Lemma~\ref{lem:cpc}, so by Corollary~\ref{cor:wb} $\epsilon \prec\!\!\!\prec \delta$ implies $\epsilon \ll \delta$. Since $S$ is algebraic, there is some compact element $k \in S$ such that $\epsilon \leqslant k \leqslant \delta$. With Lemma \ref{lem:wb}, we see that $\kappa = k^* k$ is a compact element in $\mathit{\Sigma}$, and $\epsilon \leqslant \kappa \leqslant \delta$. 
\end{proof}

Here comes the most important of our mirror properties. 

\begin{theorem}\label{thm:continuous}
Assume that $S$ is a mirror semigroup. Then $S$ is continuous (resp.\ algebraic) if and only if $\mathit{\Sigma}$ is continuous (resp.\ algebraic). 
\end{theorem}

\begin{proof}
Assume that $\mathit{\Sigma}$ is a continuous poset. 
Then $\mathit{\Sigma}$ is meet-continuous by Lemma~\ref{lem:cpc}, so $S$ is separately Scott-continuous by Proposition~\ref{prop:meetcont}. 
Let us show that $S$ is continuous. Let $s \in S$. There exists some directed subset $\mathit{\Delta}$ of $\mathit{\Sigma}$ such that 
\begin{equation}\label{eq:delta}
\bigvee \mathit{\Delta} = s^* s
\end{equation} 
and $\delta \prec\!\!\!\prec s^* s$ for all $\delta \in \mathit{\Delta}$. Let $\delta \in \mathit{\Delta}$. Then $s \delta \leqslant s$ on the one hand, and $\delta s^* s \leqslant \delta \prec\!\!\!\prec s^* s$, so $\sigma(s \delta) \prec\!\!\!\prec \sigma(s)$, on the other hand. By Lemma~\ref{lem:wb}, we have $s \delta \ll s$, for all $\delta \in \mathit{\Delta}$. Also, $S$ is separately Scott-continuous, so from Equation~(\ref{eq:delta}) we deduce that $s$ is the supremum of $s \mathit{\Delta}$, 
and this set is directed and consists of elements way-below $s$. 
This establishes the continuity of $S$. 

Assume that $\mathit{\Sigma}$ is algebraic. Let $t \ll s$ and let us show that $t \leqslant k \leqslant s$ for some compact element $k \in S$. 
By Lemma~\ref{lem:wb} we have $\sigma(t) \prec\!\!\!\prec \sigma(s)$, so there is some compact element $\kappa$ in $\mathit{\Sigma}$ such that $\sigma(t) \leqslant \kappa \leqslant \sigma(s)$. We get $s \sigma(t) \leqslant s \kappa \leqslant s$, and since $t \leqslant s$ we have $t \leqslant s \kappa \leqslant s$. The element $k = s \kappa$ satifies $k^*k = \kappa$, so that $k$ is compact in $S$ by Lemma~\ref{lem:wb}. This proves that $S$ is algebraic. 
\end{proof}

A continuous inverse semigroup with a multiplicative way-below relation is called a \textit{stably continuous inverse semigroup}. 

\begin{corollary}
Assume that $S$ is a mirror semigroup. Then $S$ is stably continuous if and only if $\mathit{\Sigma}$ is stably continuous. 
\end{corollary}

If $\epsilon \in \mathit{\Sigma}$, we write $H_\epsilon$ for the subset $\{ s \in S : s^* s = \epsilon \}$. (If $S$ is a \textit{Clifford inverse semigroup}, i.e.\ an inverse semigroup such that $s^* s = s s^*$ for all $s \in S$, then $H_\epsilon$ is the maximal subgroup of $S$ with identity element $\epsilon$.)  

\begin{theorem}\label{thm:cont2}
Assume that $S$ is an inverse semigroup such that $\mathit{\Sigma}$ is a continuous poset. Then $S$ is a mirror semigroup if and only if, for all $\epsilon \in \mathit{\Sigma}$, and each pair of distinct points $s, t \in H_\epsilon$, there exists some $\varphi \in \mathit{\Sigma}$, $\varphi \prec\!\!\!\prec \epsilon$, such that $s \varphi \neq t \varphi$. In this case, $S$ is a continuous poset. 
\end{theorem}

\begin{proof}
Assume that $S$ is a mirror semigroup, and let $\epsilon \in \mathit{\Sigma}$ and $s, t \in H_\epsilon$. Suppose that, for all $\varphi \prec\!\!\!\prec \epsilon$, we have $s \varphi = t \varphi$. Since $\mathit{\Sigma}$ is continuous, $A := \{ \varphi \in \mathit{\Sigma} : \varphi \prec\!\!\!\prec \epsilon \}$ is a directed subset of $\mathit{\Sigma}$ that admits $\epsilon$ as supremum in $\mathit{\Sigma}$. Hence, $\epsilon$ is also the supremum of $A$ in $S$, for $S$ is mirror. Moreover, we have $\varphi \varphi^* = \varphi \leqslant \epsilon = s^* s$, for all $\varphi \in A$, so we can apply Lemma~\ref{lem:distr}, which gives 
\begin{align*}
s \epsilon &= s (\bigvee A) = \bigvee (s A) = \bigvee_{\varphi \in A} (s \varphi) \\
&= \bigvee_{\varphi \in A} (t \varphi) = t (\bigvee A) = t \epsilon. 
\end{align*}
Since $s, t \in H_\epsilon$, we get $s = s s^* s = s \epsilon = t \epsilon = t t^* t = t$. 

Conversely, assume that the property given in the theorem is satisfied. We want to show that $S$ is a mirror semigroup, so let $\mathit{\Delta}$ be a directed subset of $\mathit{\Sigma}$ with a supremum $\epsilon$ in $\mathit{\Sigma}$. We want to prove that $\epsilon$ is also the supremum of $\mathit{\Delta}$ in $S$, so let $u \in S$ be an upper bound of $\mathit{\Delta}$. Then $u^* u$ is an upper bound of $\mathit{\Delta}$ in $\mathit{\Sigma}$, so that $\epsilon \leqslant u^* u$ by definition of $\epsilon$. 
This implies that $(u \epsilon)^* (u  \epsilon) = \epsilon$, so we have both $\epsilon$ and $u \epsilon$ in $H_\epsilon$. 
If we suppose that $\epsilon \neq u \epsilon$, there exists some $\varphi \in \mathit{\Sigma}$, $\varphi \prec\!\!\!\prec \epsilon$ such that $\epsilon \varphi \neq u \epsilon \varphi$, i.e.\ $\varphi \neq u \varphi$. But the fact that $\varphi \prec\!\!\!\prec \epsilon$ and the definition of $\epsilon$ imply that there is some $\epsilon_1 \in \mathit{\Delta}$ such that $\varphi \leqslant \epsilon_1$, so that $\varphi \leqslant u$, i.e.\ $\varphi = u \varphi$, a contradiction. We have thus proved that $\epsilon = u \epsilon$, which rewrites to $\epsilon \leqslant u$. This shows that $\epsilon$ is the least upper bound of $\mathit{\Delta}$ in $S$. 
\end{proof}

\begin{problem}
Find a non-continuous inverse semigroup whose semilattice of idempotents is continuous. (Note that in the inverse semigroup $S = [0, 1] \cup \{\omega\}$ given in Example~\ref{ex:cex}, the element $\omega$ is compact, thus $S$ is continuous.)
\end{problem}

\begin{example}[Example~\ref{ex:bic} continued]
In the bicyclic monoid, the semilattice of idempotents is isomorphic to $(P, \geqslant)$. In the particular cases where $P$ equals $\mathbb N$ or $\mathbb{R}_+$, which both are stably continuous semilattices, the associated bicyclic monoid is stably continuous. 
\end{example}

\begin{example}[Example~\ref{ex:rot} continued]
In the rotation semigroup, the semilattice of idempotents $[0,1]$ is stably continuous, so $(B^2, \otimes)$ 
is stably continuous. 
\end{example}

\begin{example}[Example~\ref{ex:cha} continued]
Let $S$ be a finite commutative inverse monoid. 
The cube $[0, 1]^{\mathit{\Sigma}(S)}$, as a finite cartesian product of continuous lattices, is a continuous lattice \cite[Proposition~I-2.1]{Gierz03}. 
Considering the semilattice $\mathit{\Sigma}(S)^{\wedge}$ of characters on $\mathit{\Sigma}(S)$ as a subset of $[0, 1]^{\mathit{\Sigma}(S)}$, it is closed under arbitrary infima and suprema, so it is a continuous lattice \cite[Theorem~I-2.6]{Gierz03}. 
Since $\mathit{\Sigma}(S)^{\wedge}$ and $\mathit{\Sigma}(S^{\wedge})$ are isomorphic, the semilattice $\mathit{\Sigma}(S^{\wedge})$ is also continuous. 
Now the character monoid $S^{\wedge}$ of $S$ is mirror, so $S^{\wedge}$ is continuous. 
\end{example}

The case of the symmetric pseudogroup introduced in Example~\ref{ex:sym} is extended to topological spaces as follows. 
If $X$ is a topological space, the \textit{symmetric pseudogroup} $\mathrsfs{I}(X)$ on $X$ is the set made up of all the partial homeomorphisms on $X$, i.e.\ the homeomorphisms $f : U \rightarrow V$ where $U$ and $V$ are open sets of $X$. The law of inverse semigroup is defined as in the discrete case. 
The symmetric pseudogroup is directed-complete, hence is a mirror semigroup. 

A topological space $X$ is \textit{core-compact} if its collection of closed subsets $(\mathrsfs{F}(X), \supset)$ is a continuous poset. We then have the following characterization. 

\begin{corollary}
Let $X$ be a topological space. Then $X$ is core-compact if and only if its symmetric pseudogroup $\mathrsfs{I}(X)$ is continuous. 
\end{corollary}

\begin{proof}
Let $\mathit{\Sigma}$ be the semilattice of idempotents of $\mathrsfs{I}(X)$. 
Defining the maps $i : \mathrsfs{F}(X) \rightarrow \mathit{\Sigma}$ and $j : \mathit{\Sigma} \rightarrow \mathrsfs{F}(X)$ respectively by $i(F) = \id_{X\setminus F}$ and $j(f) = X \setminus \dom(f)$, it is easy to show that both $i$ and $j$ are upper adjoints of each other, i.e.\ that $i(F) \leqslant f$ if and only if $F \subset j(f)$, and $f \leqslant i(F)$ if and only if $j(f) \subset F$, for all $F \in \mathrsfs{F}(X)$ and $f \in \mathit{\Sigma}$. 
Thus, $i$ and $j$ are both isomorphisms of complete lattices. 
By \cite[Theorem~I-2.11]{Gierz03} we deduce that $X$ is core-compact if and only if $\mathit{\Sigma}$ is continuous, and, by Theorem~\ref{thm:continuous}, if and only if $\mathrsfs{I}(X)$ is continuous. 
\end{proof}

The topologist may grant more appeal to the following corollary. 

\begin{corollary}
Let $X$ be a Hausdorff topological space. Then $X$ is locally compact if and only if its symmetric pseudogroup $\mathrsfs{I}(X)$ is continuous. 
\end{corollary}

\begin{proof}
Given that the space $X$ is Hausdorff, it is known since Hofmann and Mislove \cite{Hofmann81} that local compactness and core-compactness are equivalent properties. 
\end{proof}

\begin{corollary}
Let $X$ be a Hausdorff topological space. Then $X$ is totally disconnected and locally compact if and only if its symmetric pseudogroup $\mathrsfs{I}(X)$ is algebraic. 
\end{corollary}

\begin{proof}
See \cite[Exercise~I-4.28(iv)]{Gierz03}, where it is asserted that the lattice of closed subsets of a Hausdorff space $X$ is algebraic if and only if $X$ is totally disconnected and locally compact, then apply Theorem~\ref{thm:continuous}. 
\end{proof}

\section{Conclusion and perspectives}

The work presented in this paper is a first step in the study of inverse semigroups from a domain theoretical perspective. In future work we shall aim at topological considerations, using Scott's and Lawson's topologies. We shall also examine in more detail (compact) topological inverse semigroups. 

\begin{acknowledgements}
I would like to thank three anonymous referees, notably one of them who suggested Example~\ref{ex:cex} and the result asserted by Theorem~\ref{thm:cont2}. 
\end{acknowledgements}

\bibliographystyle{plain}

\def\cprime{$'$} \def\cprime{$'$} \def\cprime{$'$} \def\cprime{$'$}
  \def\ocirc#1{\ifmmode\setbox0=\hbox{$#1$}\dimen0=\ht0 \advance\dimen0
  by1pt\rlap{\hbox to\wd0{\hss\raise\dimen0
  \hbox{\hskip.2em$\scriptscriptstyle\circ$}\hss}}#1\else {\accent"17 #1}\fi}
  \def\ocirc#1{\ifmmode\setbox0=\hbox{$#1$}\dimen0=\ht0 \advance\dimen0
  by1pt\rlap{\hbox to\wd0{\hss\raise\dimen0
  \hbox{\hskip.2em$\scriptscriptstyle\circ$}\hss}}#1\else {\accent"17 #1}\fi}

\end{document}